\RequirePackage{fix-cm}
\documentclass[smallcondensed]{svjour3}     % onecolumn (ditto)
\smartqed  % flush right qed marks, e.g. at end of proof
\usepackage{graphicx}
\usepackage{mathptmx}      % use Times fonts if available on your TeX system
\usepackage{latexsym}
\usepackage{enumitem}
\usepackage{amsfonts}
\usepackage{amssymb}
\begin{document}

\title{Squares of real conjugacy classes in finite groups \thanks{{\it Dedicated to Professor Francisco P\'erez-Monasor on the occasion of his retirement.}}\thanks{The first and second authors are supported by the Valencian Government, Proyecto PROMETEOII/2015/011. The first and the third authors are also partially supported by Universitat Jaume I, grant P11B\-2015-77.}
}

%\titlerunning{Squares of real conjugacy classes}        % if too long for running head

\author{A. Beltr\'an \and M.J. Felipe  \and C. Melchor}

\authorrunning{A. Beltr\'an \and M.J. Felipe  \and C. Melchor} % if too long for running head

\institute{Antonio Beltr\'an \and Carmen Melchor \at
              Departamento de Matem\'aticas,\\
              Universidad Jaume I, \\
              12071 Castell\'on, Spain\\
              Tel.: +34-964-72-84-25\\
              Tel.: +34-964-72-84-28\\
              \email{abeltran@mat.uji.es\\
              cmelchor@uji.es}           %  \\
           \and
           Mar\'{\i}a Jos\'e Felipe \at
           Instituto Universitario de Matem\'atica Pura y Aplicada,\\
            Universidad Polit\'ecnica de Valencia, \\
            46022 Valencia, Spain\\
            Tel.: +34-963-87-70-00\\
             \email{mfelipe@mat.upv.es}
}

\date{Received: date / Accepted: date}

\maketitle

\begin{abstract}
We prove that if a finite group $G$ contains a conjugacy class $K$ whose square is of the form $1 \cup D$, where $D$ is a conjugacy class of $G$,  then $\langle K\rangle$ is a solvable proper normal subgroup of $G$ and we completely determine its structure. We also obtain the structure of those groups in which the assumption above is true for all non-central conjugacy classes and when every conjugacy class satisfies that its square is the union of all central conjugacy classes except at most one.
%\PACS{PACS code1 \and PACS code2 \and more}
\end{abstract}

\section{Introduction}
\label{intro}

There exist many results about the structure of a finite group focused on the product of its conjugacy classes. Some of them are related to the solvability and non-simplicity of the group. Perhaps, one of the most relevant problems was stated in 1985 by Z. Arad and M. Herzog who conjectured that in a non-abelian simple group, the product of two non-trivial conjugacy classes can never be a single conjugacy class. This conjecture is still open although it has been obtained for several families of simple groups (see \cite{Moori}). A particular case of the conjecture was recently given by G. Navarro and R. Guralnick in \cite{GuralnickNavarro}. They proved that when a conjugacy class $K$ in a finite group $G$ satisfies that $K^2$ is again a conjugacy class, then $\langle K\rangle$ is a solvable (normal) subgroup of $G$. Another result concerning products of classes was given by Z. Arad and E. Fisman \cite{AradFisman} who demonstrated that if the product of two conjugacy classes of a group $G$ is exactly the union of these two classes, then $G$ is not a non-abelian simple group.\\

In this paper we present a contribution to the study of the solvability and non-simplicity of a group from the square of a conjugacy class. Suppose that $K$ is a non-trivial real conjugacy class of $G$, that is, a class satisfying that $K^{-1}=K$.  It trivially follows that $K^2$ can never be a conjugacy class unless $K$ consists of only a single central involution of $G$. However, as $K^2$ is always a $G$-invariant set, we can write $K^2= 1 \cup A$, where $A$ is the join of conjugacy classes of $G$. In this note, we study the extreme case in which $A$ is a single class, and we wonder whether one may obtain somewhat information concerning solvability inside the group $G$. The answer is affirmative: $K$ generates a solvable (normal) subgroup and we determine its structure. Notice that every class satisfying the property of the following result needs to be a real class. In fact, this is not a very unsual situation in finite groups. \\

{\bf Theorem A.} {\it Let $K=x^G$ be a conjugacy class of a finite group $G$ and suppose that $K^2= 1 \cup D$, where $D$ is a conjugacy class of $G$. Then $\langle D\rangle=[x,G]$ is either cyclic or $p$-group for some prime $p$. Since $|\langle K\rangle/\langle D\rangle|\leq 2$, then $\langle K\rangle=\langle x\rangle[x,G]$ is solvable. More precisely,
\begin{enumerate}
\item Suppose that $|K|=2$.
\begin{enumerate}
\item If $o(x)=2$, then $\langle K\rangle \cong \mathbb{Z}_{2}\times \mathbb{Z}_{2}$ and $\mathbb{Z}_{2}\cong \langle D\rangle\subseteq$ {\rm \textbf{Z}}$(G)$.
\item If $o(x)=n>2$, then $\langle K\rangle \cong \mathbb{Z}_{n}$ and $\langle D\rangle$ is cyclic.
\end{enumerate}
\item Suppose that $|K|\geq 3$.
\begin{enumerate}
\item If $o(x)=2$, then either $\langle K \rangle$ and  $\langle D\rangle$ are 2-elementary abelian groups or $\langle D\rangle$ is a $p$-group and $|K|=p^r$ with $p$ an odd prime and $r$ a positive integer.
\item If $o(x)>2$, then $\langle D\rangle$ is a $p$-elementary abelian group for some odd prime $p$. Furthermore, either $o(x)=p$ or $o(x)=2p$.
\end{enumerate}

\end{enumerate}}

Observe that in case $2$ we determine the order of the elements of $K$, which may be either $2$, $p$ or $2p$ with $p$ an odd prime.\\

All cases of Theorem A are feasible and we provide examples of each one. Our techniques for proving Theorem A are quite elementary although we make use of Glauberman's ${\rm Z}^*$ theorem \cite{Glauberman} and a result of Y. Berkovich and L. Kazarin in \cite{BerkoKaza}. Both require tools from modular representation theory, so our results are based on it as well. Other two main ingredients of the proof of Theorem A are Burnside's classification of finite 2-groups having exactly one involution and the classification of groups of order 16. We remark that we do not appeal to the Classification of Finite Simple Groups.\\

On the other hand, we prove that the property $K^2=1\cup D$ in a group $G$, stated in Theorem A, can be characterized by means of an arithmetical property of the set of irreducible characters of $G$. As usual, Irr$(G)$ will denote this set.\\

{\bf Theorem B.} {\it Let $G$ be a group and $x, d\in G\setminus\lbrace 1\rbrace$. Let $K=x^G$ and $D=d^G$. The following are equivalent:}
\begin{enumerate}[label=\alph*)]
\item $K^2=1\cup D$
\item For every $\chi \in$ {\rm Irr}$(G)$ we have $$|K|\chi(x)^2=\chi(1)^2+(|K|-1)\chi(1)\chi(d).$$
\end{enumerate}

As an application of our main result we obtain the following corollaries. The first one is related to groups in which every non-central conjugacy class satisfies the hypothesis of Theorem A and its proof is a trivial consequence. The second concerns conjugacy classes whose square is the join of classes all central except at most one. \\

{\bf Corollary C.} {\it Let $G$ be a finite group such that every non-central conjugacy class $K$ satisfies that $K^2= 1 \cup D$, where $D$ is a conjugacy class of $G$. Then $G/${\rm \textbf{F}}$(G)$ is an elementary abelian 2-group.}\\

{\bf Corollary D.} {\it Let $K$ be a conjugacy class of a finite group $G$ such that $K^2$ is union of conjugacy classes all of them central except at most one. Then $\langle K\rangle$ is solvable.}\\

Suppose now that every conjugacy class $K$ of a group $G$ satisfies that $K^2$ is a conjugacy class. It is trivial that every real element must lie in ${\bf Z}(G)$ and must have order 2. In \cite{ChillagMann}, D. Chillag and A. Mann described the groups in which every real element is a central element. Particularly, in Remark 5.5 of  \cite{ChillagMann}, the authors also announce, with omitted proof,  that any group satisfying this property is nilpotent. We include here an extension of this result which will be needed in order to study the structure of those groups in which all conjugacy classes satisfy the hypothesis of Corollary C. Note that these groups are solvable by Theorem A and we show that they are close to nilpotent groups.\\

{\bf Corollary E.} {\it Let $G$ be a finite group and let $\pi$ be a set of primes. Suppose that $K^2$ is a conjugacy class for all conjugacy class $K$ of $\pi$-elements of $G$. Then $G/${\rm \textbf{O}}$_{\pi'}(G)$ is nilpotent. In particular, if $\pi=\pi(G)$, then $G$ is nilpotent.}\\

{\bf Corollary F.} {\it Let $G$ be a finite group such that every conjugacy class $K$ satisfies that $K^2$ is union of conjugacy classes all of them central except at most one. Let $M/${\rm\textbf{F}}$(G)=$ {\rm \textbf{O}}$_{2}(G/${\rm\textbf{F}}$(G))$. Then $G/M$ is nilpotent and, consequently, $G$ is solvable with Fitting length at most 3.}

\section{Preliminary results}
We begin by stating the ${\rm Z}^*$ theorem version appearing in \cite{Huppert}.

\begin{theorem}
Let $G$ be a finite group. Suppose that $P\in $ {\rm Syl}$_{2}(G)$ and $j\in P$ such that $j^2=1\neq j$ and $P\cap \lbrace j^g \vert g\in G\rbrace=\lbrace j \rbrace$. Then {\rm \textbf{O}}$_{2'}(G)\langle j \rangle \unlhd G$.
\end{theorem}

The following elementary properties will be often used in the proofs. The reader interested in further properties related to products and powers of conjugacy classes can refer to \cite{products}.

\begin{lemma}
Let $K$ be a real conjugacy class of a finite group $G$. Then
\begin{enumerate} [label=\alph*)]
\item $\langle K \rangle /\langle K^2\rangle$ is trivial or cyclic of order 2.
\item If $K^2=1\cup K$, then $\langle K\rangle$ is a minimal normal subgroup of $G$ and it is p-elementary abelian for some prime $p$.
\end{enumerate}
\end{lemma}

\begin{proof}
We write $N=\langle K^2 \rangle$ and consider the factor group $\langle K \rangle/N$ which is generated by the set of elements $xN$ with $x\in K$. But, if $x, y \in K$ we know that $xy^{-1} \in K^2 \subseteq N$, so $xN=yN$, and consequently $\langle K\rangle/N$  is generated by just one element, say $xN$, with $x\in K$, and $(xN)^2 =N$. Thus, a) is proved. Statement b) trivially follows because $K^2$ is a subgroup of $G$ and all non-trivial elements of $K^2$ are conjugate, so all of them have the same order.\qed
\end{proof}

As we have already indicated in the Introduction, we use the following result of Berkovich and Kazarin of \cite{BerkoKaza}, which is based on the well-known Kazarin's Theorem (see for instance \cite{Huppert}), which asserts that any conjugacy class of prime-power size generates a solvable normal subgroup.

\begin{lemma}[Lemma 3 of \cite{BerkoKaza}]
Let $G$ be a finite group and let $x\in G$. If $|x^G|$ is a power of $q\in \pi(G)$, then $(\langle x\rangle^{G})'$ is a $q$-subgroup. In particular, $\langle x\rangle^G/ ${\rm \textbf{O}}$_{q}(\langle x\rangle^G)$ is an abelian $\pi(o(x))$-group.
\end{lemma}

The following result, to which we have referred at the beginning of the Introduction, will be used in the proof of Corollaries C and D. We remark that we do need it to prove Theorem A.

\begin{theorem}[Theorem A of \cite{GuralnickNavarro}]
Let $G$ be a finite group, let $x\in G$, and let $K=x^G$ be the conjugacy class of $x$ in $G$. Then the following are equivalent:
\begin{enumerate}[label=\alph*)]
\item $K^{2}$ is a conjugacy class of $G$.
\item $K=x[x,G]$ and {\rm \textbf{C}}$_{G}(x)=$ {\rm \textbf{C}}$_{G}(x^2)$.
\end{enumerate}
In this case, $[x,G]$ is solvable. Furthermore, if $x$ has order a power of a prime $p$, then $[x,G]$ has a normal $p$-complement.
\end{theorem}

The original result of \cite{GuralnickNavarro} includes one more assertion related to Character Theory but we do not use it in this paper. Furthermore, the proof of the equivalence between a) and b) in the above theorem, although is omitted here, can be easily obtained without using characters. The solvability in Theorem 2, however, needs the Classification of Finite Simple Groups.\\

Finally, we will also use a Burnside's classic result whose proof can be found in \cite{Michler} for instance.

\begin{lemma}[Theorem 1.2.6 of \cite{Michler}]
A non-cyclic 2-group $P$ has only one involution if and only if $P$ is a generalized quaternion group.
\end{lemma}

\section{Proofs}
We start by proving the equalities concerning  commutators that appear in Theorem A. Recall that the conjugacy classes $K$ and $D$ of the statement are real classes.

\begin{lemma}
Let $K=x^G$ be a conjugacy class of a finite group $G$ and suppose that $K^2= 1 \cup D$, where $D$ is a conjugacy class of $G$. Then $\langle D\rangle=[x,G]$ and $\langle K\rangle=\langle x\rangle[x, G]$.
\end{lemma}

\begin{proof}
If $K=\lbrace x_{1}, \ldots, x_{n}\rbrace$, it is clear that $K^{2}=x_{1}K\cup \cdots \cup x_{n}K$. Let $y \in x_{i}K$. Since $K$ is real, then $y=x_{i}g^{-1}x_{i}^{-1}g \in [x_{i}^{-1}, G]=[x_{i}, G]$ for some $g\in G$. Also, if $i\neq j$, then $x_{j}=x_{i}^{h}$ for some $h\in G$. Thus, $[x_{j}, G]=[x_{i}^{h}, G]=[x_{i}, G]^{h}=[x_{i}, G]$. Consequently, $K^{2}\subseteq [x, G]$ and $\langle D\rangle\subseteq [x, G]$. On the other hand, since any element $[x,t]$ lies in $K^2$ for all $t\in G$, then $[x, G]\subseteq \langle K^{2}\rangle=\langle D\rangle$ and hence, $\langle D\rangle=[x, G]$. The equality $\langle K\rangle=\langle x\rangle[x, G]$ is standard, since the normal closure $\langle x\rangle^{G}$ of a subgroup $\langle x\rangle$ is equal to $\langle x\rangle[x, G]$.\qed
\end{proof}

We are ready to prove our main result.

\begin{proof}[of Theorem A]
The proof is divided into two cases: when $|K|=2$ and when $|K|\geq 3$. \\

Case 1: Suppose that $|K|=2$.\\

Case 1.a. Let $K=\lbrace x, x^g\rbrace$ with $g\in G$. If $o(x)=2$, then $K^{2}=1 \cup \lbrace xx^g, x^gx\rbrace$. But observe that since ${\rm \textbf{C}}_{G}(x)\unlhd G$, we have ${\rm \textbf{C}}_{G}(x)={\rm \textbf{C}}_{G}(x^{g})$, so $xx^g=x^gx$. Thus, $K^{2}=1 \cup \lbrace xx^g\rbrace$, so $xx^g\in {\rm \textbf{Z}}(G)$ and $\langle D \rangle=\langle xx^g\rangle\subseteq {\rm \textbf{Z}}(G)$. Furthermore, $\langle K \rangle=\langle x, x^g\rangle=\langle x\rangle\times \langle x^g\rangle\cong \mathbb{Z}_{2}\times \mathbb{Z}_{2}$ and $|\langle K\rangle/\langle D\rangle|=2$. So, $1.a$ is proved.\\

Case 1.b. If $o(x)=n>2$, then $K=\{ x, x^{-1}\}$ and as a consequence, $\langle K \rangle=\langle x\rangle$, which is cyclic of order $n$ and $\langle D\rangle=\langle x^2\rangle$. If $n$ is odd, then $\langle D\rangle=\langle x\rangle=\langle K\rangle$. If $n$ is even, then $|\langle K\rangle/\langle D\rangle|=2$.  Therefore, $1.b$ is proved.\\

Case 2: Suppose that $|K|\geq 3$.\\

Case 2.a. Suppose that $o(x)=2$ and let $t\in D$. We distinguish two cases depending on the order of $t$. Suppose that $o(t)=2$. We set $K=\lbrace x_{1}, \ldots, x_{s}\rbrace$ and we have $K^{2}=1\cup \lbrace x_{i}x_{j}\vert i\neq j \rbrace$. Since $o(x_{i}x_{j})=2$ for every $1 \leq i, j \leq s$ with $i \neq j$, then $1=x_{i}^{2}x_{j}^{2}=x_{i}x_{i}x_{j}x_{j}=x_{i}x_{j}x_{i}x_{j}$, so $x_{i}$ and $x_{j}$ commute. Consequently, $\langle K \rangle$ is generated by pairwise commuting involutions, so $\langle K\rangle$ is 2-elementary abelian (and $\langle D\rangle$ too) and we obtain the first assertion of $2.a$.\\

Suppose now that $o(t)>2$. Observe that any two distinct elements of $K$ do not commute. Otherwise, the order of $t$ would be necessarily 2. As a consequence, each $x_{i}\in K$ acts via conjugation on $K$ in such a way that it fixes only the element $x_{i}$ and permutes in pairs the elements of $K\setminus \lbrace x_{i}\rbrace$. As a result, $|K|$ is odd. This implies that $x \in {\rm \textbf{Z}}(P)$ for some $P\in {\rm Syl}_{2}(G)$. Therefore, $P \cap K=\lbrace x\rbrace$ and we deduce that ${\rm \textbf{O}}_{2'}(G)\langle x \rangle \unlhd G$ by Theorem 1.\\

The Frattini argument gives $G={\rm \textbf{N}}_{G}(\langle x\rangle){\rm \textbf{O}}_{2'}(G)$. But observe that ${\rm \textbf{N}}_{G}(\langle x\rangle)={\rm \textbf{C}}_{G}(x)$ because $o(x)=2$. Thus, $G={\rm \textbf{C}}_{G}(x){\rm \textbf{O}}_{2'}(G)$. As a result, $$K^{2}=\lbrace x^{-1}x^{g} \vert x\in K, g\in G\rbrace=\lbrace [x,g] \vert x\in K, g\in {\rm \textbf{O}}_{2'}(G) \rbrace\subseteq {\rm \textbf{O}}_{2'}(G).$$ Then $D\subseteq {\rm \textbf{O}}_{2'}(G)$ and in particular, $|\langle D\rangle|$ is odd.\\

Now, we prove that $|K|$ is a power of an odd prime. Since $x\not \in {\rm \textbf{Z}}(G)$, we can take an odd prime $p$ dividing $|{\rm \textbf{O}}_{2'}(G): {\rm \textbf{C}}_{{\rm \textbf{O}}_{2'}(G)}(x)|$. As ${\rm \textbf{O}}_{2'}(G)$ has odd order, then the number of Sylow $p$-subgroups of ${\rm \textbf{O}}_{2'}(G)$ is also odd, and hence $x$, which acts on this set of subgroups, must fix one of them. Let $P\in$ Syl$_{p}({\rm \textbf{O}}_{2'}(G))$ such that $P^x=P$. Thus, $[x, P]\subseteq P$. Now, if $[x,P]=1$ this contradicts that $p$ divides $|{\rm \textbf{O}}_{2'}(G): {\rm \textbf{C}}_{{\rm \textbf{O}}_{2'}(G)}(x)|$, so $[x,g]$ is a non-trivial $p$-element for some $g\in P$. Therefore, $[x,g]=xx^g\in K^{2}$ is a $p$-element lying in $D$, so all elements of $D$ are $p$-elements and in particular, the prime $p$ is unique. Moreover, $$p^m=|{\rm \textbf{O}}_{2'}(G): {\rm \textbf{C}}_{{\rm \textbf{O}}_{2'}(G)}(x)|=|{\rm \textbf{C}}_{G}(x){\rm \textbf{O}}_{2'}(G): {\rm \textbf{C}}_{G}(x)|=|G:{\rm \textbf{C}}_{G}(x)|=|K|$$ for some $m\geq 1$, as we wanted to prove. By applying Lemma 2, we deduce that $\langle K\rangle /{\rm \textbf{O}}_{p}(\langle K\rangle)$ is an abelian $2$-group. As a consequence, ${\rm \textbf{O}}_{2'}(\langle K\rangle)={\rm \textbf{O}}_{p}(\langle K\rangle)$. By Lemma 1(a), we have that $\langle K\rangle/\langle D\rangle$ is necessarily cyclic of order 2 and then $\langle D\rangle={\textbf{O}}_{p}(\langle K\rangle)$. The second assertion of case 2.a is complete. \\

Case 2.b. Suppose that $o(x)>2$. We prove first that $|D|=|K|$. We know that $|K|\leq |K^2|=1+|D|$ and $D=(x^2)^{G}$. Note that $|D|$ divides $|K|$ because ${\rm \textbf{C}}_{G}(x)\subseteq {\rm \textbf{C}}_{G}(x^2)$. Thus, either $|D|=|K|$ or $|D|\leq |K|/2$. However, if $|D|\leq |K|/2$, the first inequality implies that $|K|\leq 1+|K|/2$, so $|K|\leq 2$, a contradiction. Consequently, $|K|=|D|$ as wanted. Now, notice that $xK\cup x^{-1}K\subseteq K^2$ and we claim that $xK\neq x^{-1}K$. Indeed, if $xK=x^{-1}K$, then $x^2K=K$. Hence for all $g\in G$, it follows that $(x^g)^2K=(x^2K)^g=K^g=K$, which means that $DK=K$. As a result, $\langle D\rangle K=K$. This implies that $K$ is union of right classes of $\langle D\rangle$ and then $|\langle D \rangle|$ divides $|K|$. But this is a contradiction because $|K|=|D|<|\langle D\rangle|$.\\

We have proved that $xK\neq x^{-1}K$ with $xK\cup x^{-1}K\subseteq K^2$. Since $|K^2|=|K|+1$ and $|K|=|xK|=|x^{-1}K|$, there exists only just one element $z\in xK\setminus x^{-1}K$. Moreover, it is easy to prove that $z^{-1}$ is the only element contained in $x^{-1}K\setminus xK$ (notice that $z$, $z^{-1}\in D$). Therefore, $K^2$ can be decomposed as $$K^{2}=xK\cup x^{-1}K=(xK\cap x^{-1}K)\cup \lbrace z\rbrace\cup \lbrace z^{-1}\rbrace.$$

From the fact that $(xK)(x^{-1}K)=K^2$ and $K^{4}=(1\cup D)(1\cup D)=K^{2}\cup D^{2}$, we deduce that
$$K^4=K^2\cup \lbrace z^{2}\rbrace\cup\lbrace z^{-2}\rbrace=1 \cup D\cup \lbrace z^{2}\rbrace\cup\lbrace z^{-2}\rbrace.$$

Let us see that $K^4=D^2$. We know that $D^2$ is a $G$-invariant set, so we can write $D^{2}=1\cup A_{1}\cup \cdots \cup A_{r}$ with $A_{i}$ a conjugacy class for $1\leq i \leq r$. On the other hand, since $xK\subseteq K^2=1\cup D$ then $xK=1\cup D'$ with $D'\subseteq D$ and similarly $x^{-1}K=1\cup D''$ with $D''\subseteq D$. Thus, $D'D''\subseteq K^{2}\cap D^{2}$ and $|D'D''|\geq |D'|=|K|-1\geq 2$. We conclude that there exists $1\neq g\in K^{2}\cap D^{2}$. As a result, $g\in D$. Also, $g\in A_{i}$ for some $1\leq i \leq r$. Consequently, $D=A_{i}$ and hence $D\subseteq D^{2}$. Accordingly, $K^{4}=1\cup D\cup D^{2}=D^{2}$, as wanted. Therefore, $$D^{2}=1 \cup D\cup \lbrace z^{2}\rbrace\cup\lbrace z^{-2}\rbrace.$$ We distinguish two subcases depending on whether $z^2 \in K^2$ or $z^2 \not \in K^2$.

\begin{enumerate}[label=\alph*), leftmargin=5mm]
\item If $z^2 \in K^2$, we have either $z^2=1$ or $z^2\in D$ (and $z^{-2}\in D$). In both cases, it follows that $D^2=K^2=1\cup D$, and then $\langle D\rangle$ is $p$-elementary abelian for some prime $p$ by applying Lemma 1(b). Furthermore, $\langle D\rangle=\langle K^2\rangle$, so $|\langle K\rangle/ \langle D\rangle|\leq 2$ by Lemma 1(a). Observe that $(x^2)^G=D$, so $o(x)$  divides $2p$ and hence, either $o(x)=p$ or $o(x)=2p$. If $o(x)=p>2$, then $\langle x\rangle =\langle x^2\rangle$ and $\langle K\rangle=\langle D\rangle$ is $p$-elementary abelian. Let us prove that if $o(x)=2p$, then $p$ is odd. If $p=2$, then $o(x)=4$ and we know that $2^{a}=|\langle D\rangle|=1+|D|$ for some $a>1$. Thus, $|D|=2^{a}-1=|K|$ is odd. As $K$ is real, then $o(x)=2$, a contradiction. Thus, we obtain the assertion of 2.b.\\

\item Suppose that $z^2\not \in K^2$, what is equivalent to say that either $\lbrace z^{2}\rbrace$ and $\lbrace z^{-2}\rbrace$ are central classes or $\lbrace z^{2}, z^{-2}\rbrace$ is a single conjugacy class of cardinality $2$. The rest of the proof consists in getting a contradiction by a series of steps. \\

\textit{Step 1: $\langle D\rangle/\langle z^2\rangle$ is a $2$-elementary abelian group. Moreover, $\langle D\rangle$ is nilpotent of class at most 2. Therefore, we write $\langle D\rangle=P\times H$ with $P\in$ {\rm Syl}$_{2}(\langle D\rangle)$ and $H$ a $2$-complement of $\langle D\rangle$ with $H\subseteq \langle z^2\rangle \subseteq$ {\rm \textbf{Z}}$(\langle D\rangle)$.}\\

By the hypotheses of $b)$, it follows that $\langle z^2\rangle\unlhd G$. We denote $\overline{G}=G/\langle z^2\rangle$ and consider $\overline{D}$. We have $\overline{D}^2=\overline{D^2}=\overline{1}\cup \overline{D}$. By Lemma 1(b), $\langle \overline{D}\rangle$ is $p$-elementary abelian for some prime $p$. Observe that if $d\in D$, then $d=z^g$, for some $g\in G$, and $d^2=(z^g)^2=(z^2)^g\in \lbrace z^2, z^{-2}\rbrace$. Thus, $\overline{d}^2=\overline{d^2}=1$ and $p=2$. Furthermore, $z\in {\rm\textbf{C}}_{G}(z^2)\unlhd G$ and hence, $D \subseteq {\rm \textbf{C}}_{G}(z^2)$. This means that $\langle z^2\rangle \subseteq {\rm \textbf{Z}}(\langle D\rangle)$, so $\langle D\rangle/{\rm \textbf{Z}}(\langle D \rangle)$ is abelian and $\langle D\rangle$ is nilpotent of class at most $2$. The decomposition for $\langle D\rangle$ of the statement certainly holds. \\

\textit{Step 2: We can assume that $o(z^2)$ is even.}\\

If $z^2\in {\rm \textbf{Z}}(G)$, since $z\in D$ and $D$ is real, we have that $z^2$ and $z^{-2}$ are conjugate and hence $z^2=z^{-2}$. Thus, $o(z^2)=2$. Consequently, we can assume that $\lbrace z^{2}, z^{-2}\rbrace$ is a conjugacy class of cardinality $2$ for the rest of this step. Suppose that $o(z^2)=k$ is odd and notice that  $o(z)=2k$, so we can write $z=z^{k}z^2$ where $z^k$ and $z^2$ are the $2$-part and the $2'$-part of $z$, respectively. Moreover, there exists $g\in G$ such that $(z^2)^g=z^{-2}$. We know that $zz^g\in D^2=1\cup D\cup \lbrace z^2, z^{-2}\rbrace$. Also, $zz^g=z^kz^2(z^k)^gz^{-2}=z^k(z^k)^g\in P$ by taking into account that $P\unlhd G$, so $zz^g$ is a $2$-element. As a consequence, $zz^g$ can only be equal to $1$, $z^2$ or $z^{-2}$ because the elements of $D$ have order $2k$. Now, if $zz^g=1$, then $z^{-1}=z^g=(z^k)^g(z^2)^g=(z^k)^gz^{-2}$ what means that $z=(z^k)^g$, a contradiction. If $zz^g$ is equal to either $z^2$ or $z^{-2}$ we can easily compute that $o(z)=4$, again a contradiction. Thus, $o(z^2)$ must be even.\\

\textit{Step 3: $\langle D\rangle=\langle z^2\rangle \cup D\langle z^2\rangle$ and $\langle D\rangle$ has just one element of order $2$ that is the involution of $\langle z^2\rangle $.}\\

Since $D\subseteq D^2$, it is easy to prove by induction on $k$ that for every $k\geq 2$, $D^{k-1}\subseteq D^{k}\subseteq 1\cup D \cup \langle z^2\rangle\cup D\langle z^2\rangle$. We can deduce that there exists $l\in \mathbb{N}$, depending on the order of $z$, such that $\langle D\rangle =D^{l}\subseteq 1\cup D \cup \langle z^2\rangle\cup D\langle z^2\rangle\subseteq \langle D \rangle$. This yields to $$\langle D\rangle=1\cup D \cup \langle z^2\rangle\cup D\langle z^2\rangle=\langle z^2\rangle \cup D\langle z^2\rangle.$$

Accordingly, it is enough to show that there exists no element $dz^{2i}\in Dz^{2i}$ with $d\in D$ such that $o(dz^{2i})=2$. Otherwise, we assume $(dz^{2i})^2=d^2z^{4i}=1$ and notice that $d^2=(z^g)^2=(z^2)^g\in \lbrace z^{2}, z^{-2}\rbrace$ for some $g\in G$. Consequently, either $z^{4i+2}=z^{2(2i+1)}=1$ or $z^{4i-2}=z^{2(2i-1)}=1$. In both cases, $o(z^2)$ would be odd, which contradicts Step 2. As a result, the unique element of order $2$ in $\langle D\rangle$ is the involution of $\langle z^2\rangle$.\\

\textit{Step 4: Final contradiction.}\\

By Step 3 and Lemma 3, $P$ must be cyclic or generalized quaternion. We will get a contradiction in both cases. Assume first that $P$ is cyclic. Since $\langle \overline{D}\rangle=\langle D\rangle/ \langle z^2 \rangle\cong P/P\cap \langle z^2\rangle$ is $2$-elementary abelian by Step 1 and $P$ is cyclic, then either $\langle \overline{D}\rangle\cong \mathbb{Z}_{2}$ or $\langle \overline{D}\rangle$ is trivial. Furthermore, $\langle z\rangle\neq \langle z^2\rangle$. Otherwise, ${\rm \textbf{C}}_{G}(z)={\rm \textbf{C}}_{G}(z^2)$ and then $|K|=|D|=|(z^2)^G|$ would be either $1$ or $2$, contradicting the fact that $|K|\geq 3$. Thus, $\langle z^2\rangle\subset\langle z\rangle \subseteq\langle D\rangle$ and this forces that $\langle D\rangle=\langle z\rangle$. As the elements of $D$ have the same order that $z$, which is even, this equality implies that they are odd powers of $z$ and, as a consequence, the elements of $D^2$ are even powers of $z$. This contradicts that $D\subseteq D^2$.\\

From now on, $P$ can be assumed to be generalized quaternion. We denote $\widetilde{G}=G/H$ and then $\langle \widetilde{D}\rangle=\widetilde{\langle D\rangle}\cong P$. Notice that ${\rm \textbf{Z}}(P)\cong {\rm \textbf{Z}}(\langle \widetilde{D}\rangle)={\rm \textbf{Z}}(\langle D\rangle/H)={\rm \textbf{Z}}(\langle D\rangle)/H$, because $\langle D\rangle=P\times H$, and that $\langle \widetilde{D}\rangle/{\rm \textbf{Z}}(\langle \widetilde{D}\rangle)$ is dihedral. Also, $\langle \widetilde{D}\rangle/{\rm \textbf{Z}}(\langle \widetilde{D}\rangle)\cong \langle D\rangle/{\rm\textbf{Z}}(\langle D\rangle)$ is $2$-elementary abelian by Step 1. By joining both facts, we conclude that $\langle \widetilde{D}\rangle/\langle {\rm \textbf{Z}}(\widetilde{D})\rangle \cong \mathbb{Z}_{2}\times \mathbb{Z}_{2}$. Therefore, $\langle \widetilde{D}\rangle\cong Q_{8}$ and $\langle D\rangle \cong Q_{8}\times H$. Note that $\langle  \widetilde{z}^2\rangle\subseteq {\rm \textbf{Z}}(\langle \widetilde{D}\rangle)$ which has order 2. Then the order of $\langle  \widetilde{z}^2\rangle$ is either $1$ or $2$. If $\langle  \widetilde{z}^2\rangle$ is trivial, from Step 3, we get $$\langle \widetilde{D}\rangle=\langle \widetilde{z}^2\rangle \cup \widetilde{D}\langle  \widetilde{z}^2\rangle=\widetilde{1}\cup \widetilde{D}$$ and by Lemma 1(b), $\langle \widetilde{D}\rangle$ is elementary abelian, a contradiction. Thus, we can assume that $o(\widetilde{z}^2)=2$. Again, by Step 3, $$\langle \widetilde{D}\rangle=\widetilde{1}\cup \widetilde{D}\cup \widetilde{D}\widetilde{z}^2 \cup \{ \widetilde{z}^2\},$$
and we distinguish two cases. When $\widetilde{D}$ and $\widetilde{D}\widetilde{z}^2$ are equal or distinct. If $\widetilde{D}\neq \widetilde{D}\widetilde{z}^2$, then $8=2+2|\widetilde{D}|=2(1+|\widetilde{D}|)$, what means that $|\widetilde{D}|=3$. The fact that $\widetilde{D}$ is real forces that $o(\widetilde{z})=2$, a contradiciton. Therefore, $\widetilde{D}=\widetilde{D}\widetilde{z}^2$ and $8=|\langle \widetilde{D}\rangle|=2+|\widetilde{D}|$, so $|\widetilde{D}|=6$. Now we prove that $|\widetilde{K}|=|\widetilde{D}|$. We have $\widetilde{K}^2=\widetilde{1}\cup \widetilde{D}$. Since $o(\widetilde{z}^2)=2$ and $\widetilde{x}^2$ and $\widetilde{z}$ are $\widetilde{G}$-conjugate, we know that $o(\widetilde{x})>2$. So, $\widetilde{D}=(\widetilde{x}^2)^{\widetilde{G}}$ and, since ${\rm \textbf{C}}_{\widetilde{G}}(\widetilde{x})\subseteq {\rm \textbf{C}}_{\widetilde{G}}(\widetilde{x}^2)\subseteq \widetilde{G}$, we conclude that $|\widetilde{D}|$ divides $|\widetilde{K}|$. As $6=|\widetilde{D}|\leq |\widetilde{K}|\leq |\widetilde{K}^2|=1+|\widetilde{D}|$, we get $|\widetilde{D}|=|\widetilde{K}|$, as wanted.\\

On the other hand, by taking into account that $\widetilde{K}$ is a real class and Lemma 1(a), we know that $\langle \widetilde{K}\rangle/\langle \widetilde{K}^2\rangle=\langle \widetilde{K}\rangle/\langle \widetilde{D}\rangle$ is trivial or cyclic of order $2$. In the former case, that is, if $\langle \widetilde{K}\rangle=\langle \widetilde{D}\rangle=\widetilde{1}\cup \lbrace \widetilde{z}^2\rbrace\cup \widetilde{D}$, we see that this leads to a contradiction. We know that $\widetilde{1}\neq \widetilde{x}\in \langle \widetilde{K}\rangle$. If $\widetilde{x}=\widetilde{z}^2$, then $o(\widetilde{x})=2$. Then $x^2\in H$, which implies that $z\in H$ and $o(z^2)$ is odd, a contradiction. So $\widetilde{x}\in \widetilde{D}$ and we can write $x=dh$ with $d\in D$ and $h\in H\subseteq \langle z^2\rangle\subseteq {\rm \textbf{Z}}(\langle D\rangle)$. Then $x^2=d^2h^2\in \langle z^2\rangle$ and we conclude that $z=(x^2)^g\in \langle z^2\rangle$ and, as a result, $\langle z\rangle=\langle z^2\rangle$. So, ${\rm \textbf{C}}_{G}(z)={\rm \textbf{C}}_{G}(z^2)$ and $|D|=|(z^2)^G|=2$, which contradicts that $|K|\geq 3$. We can assume then that $\langle \widetilde{K}\rangle/\langle \widetilde{D}\rangle \cong \mathbb{Z}_{2}$. Therefore, $\langle \widetilde{K}\rangle $ is  a 2-group of order 16, which has a normal subgroup isomorphic to $Q_8$, and moreover,  $\langle \widetilde{K}\rangle$ possesses  at least  6 elements of order 8 (the elements of $\widetilde{K}$). However, the only groups of order 16 having a normal subgroup isomorphic to $Q_8$ are: $SD_{16}$, the semidihedral group; $Q_{16}$, the generalized quaternion group; the central product of $D_8$ and $\mathbb{Z}_4$; and the direct product $Q_8 \times \mathbb{Z}_2$. The former two groups posses exactly $4$ elements of order $8$ and the latter two groups have no elements of order $8$. All cases give a contradiction.\qed
\end{enumerate}
\end{proof}

{\bf Examples} Let us show several examples of each case of Theorem A. In some of them, we use the {\sc SmallGroups}  library of {\sf GAP} \cite{GAP}. The $m$-th group of order $n$ in this library is identified by $n\#m$. \\

{\it Case 1.a.}  We take $G$ the group $$M_{2^{n+1}}=\langle a, \, \, b  \, \,\vert  \, \, a^{2^n}=b^2=1, \, \, a^b=a^{2^{n-1}+1}\rangle$$ with $n\geq 3$. We consider the conjugacy class $K=b^G$ that satisfies $K^2=1\cup D$ where $D=(a^{2^{n-1}})^{G}$.\\

{\it Case 1.b.} We consider $G=D_{2n}=\langle a, \, \, b  \, \,\vert  \, \, a^{n}=b^2=1, \, \, a^b=a^{-1}\rangle $ with $n\geq 3$ and $K=a^G$. Then $K^{2}=1\cup D$ where $D=(a^2)^G$. Remark that if $n$ is odd, then $\langle D\rangle=\langle K\rangle$ whereas if $n$ is even, $|\langle K\rangle/ \langle D\rangle|=2$.\\

{\it Case 2.a.} Let $N=\langle x_{1}\rangle \times \cdots \times \langle x_{r} \rangle=\mathbb{Z}_{2}\times \cdots \times \mathbb{Z}_{2}$ and consider the natural action of $S_{r}$ on $N$, that is, $G=NS_{r}$ is the wreath product of $N$ and $S_{r}$. In this case, $K=\lbrace x_{1}, \cdots, x_{r}\rbrace$ is a conjugacy class of $G$ such that $K^{2}=1 \cup D$ where $D=\lbrace x_{i}x_{j}\vert i\neq j\rbrace$ is a conjugacy class, because $S_{r}$ acts transitively on $D$, and $o(x_{i}x_{j})=2$ for every $i\neq j$ and $|K|=r$. This is an example of case 2.a of Theorem A in which $\langle D\rangle$ is $2$-elementary abelian and $|\langle K\rangle/\langle D\rangle|=2$.\\

The alternating group $A_{4}$ is another example where $K$ is the conjugacy class of involutions. In this case, $\langle K\rangle=\langle D\rangle$.\\

Let $G=216\#88=\langle a,\, \, b, \, \, c\, \, \mid c^3=1, \, \, a^4=1, \, \, a^2=b^2, \, \, a^b=a^{-1}, \, \, c^{-1}a^{-1}ca^{-1}bcb=1, \, \,$\\ $c^{-1}a^{-1}c^{-1}b^{-1}ca^{-1}b=1, \, \, b^{-1}ca^{-1}c^{-1}abc=1\rangle \cong ((\mathbb{Z}_3 \times \mathbb{Z}_3) \rtimes \mathbb{Z}_3) \rtimes Q_8.$ The conjugacy class $K=(a^2)^G$ satisfies that $K^2=1 \cup D$ where $D=c^G$. Moreover, $o(a^2)=2$, $o(c)=3$, $|K|=9$ and $|D|=24$. This is an example of case 2.a in which $\langle D\rangle$ is a non-abelian extraspecial $3$-group of order $27$ and exponent 3. \\

{\it Case 2.b.} Let  $\langle a\rangle \cong \mathbb{Z}_{5}$ and let $\langle b\rangle \cong \mathbb{Z}_{8}$ acting on $\langle a\rangle$ by $a^b=a^2$. Let $G$ be the associated semidirect product $\langle a\rangle \rtimes \langle b\rangle$ and take $K=(b^4a)^G$. We have $K^2=1\cup D$ where $D=a^G$, $o(b^4a)=10$, $o(a)=5$, $|K|=4$ and $|D|=4$. This shows case 2.b of Theorem A with $\langle D\rangle\cong \mathbb{Z}_{5}$ and $\langle K\rangle\cong \mathbb{Z}_{10}$.\\

We get another example for the case in which the order of the elements of $K$ is a prime. Take $G=(\mathbb{Z}_3 \times \mathbb{Z}_3) \rtimes Q_8\cong {\rm PSU}(3,2)=\langle a, \, \, b, \, \, c, \, \, d\, \, | a^4=c^3=d^3=1, \, \, a^2=b^2, \, \, a^b=a^{-1}, \, \, c^a=cd^2, \, \, d^a=c^2d^2, \, \, c^b=d, \, \, d^b=c^2\rangle$ and $K=c^{G}$, with $c$ an element of order $3$. This class satisfies that $K^2=1\cup K$ with $|K|=8$. Furthermore, $\langle K\rangle \cong \mathbb{Z}_3\times \mathbb{Z}_3$. \\

Observe that both examples satisfy $|K|=|D|$, as it is explicitely showed in the proof of Theorem A.\\

Now we prove the characterization of the property stated in Theorem A in terms of irreducible characters. For our purposes, we use the following result which characterizes when the product of two conjugacy classes is again a conjugacy class.

\begin{lemma}
Let $G$ be a group and let $a, b, c \in G$ be nontrivial elements of $G$. The following conditions are equivalent:
\begin{enumerate}[label=\alph*)]
\item $a^Gb^G=c^G$
\item $\chi(a)\chi(b)=\chi(c)\chi(1)$ for all $\chi \in$ Irr$(G).$
\end{enumerate}
\end{lemma}

\begin{proof}
See for instance Lemma 1 of \cite{Moori}.\qed
\end{proof}

\begin{proof}[of Theorem B]
Suppose that $K^2=1\cup D$ and let $\chi \in$ Irr$(G)$. Notice that $K$ is real. By applying problem 3.12 of \cite{Isaacs}, $$\chi(x)^2=\chi(x)\chi(x^{-1})=\frac{\chi(1)}{|G|}\sum_{g\in G}\chi(x(x^{-1})^g).$$
We can divide the sum into two parts, so the above formula is equal to
$$\frac{\chi(1)}{|G|}(\sum_{g\in {\rm \textbf{C}}_{G}(x)}\chi(x(x^{-1})^g)+\sum_{g\in G\setminus {\rm \textbf{C}}_{G}(x)}\chi(x(x^{-1})^g))=$$
$$\frac{\chi(1)}{|G|}(|{\rm \textbf{C}}_{G}(x)|\chi (1)+(|G|-|{\rm \textbf{C}}_{G}(x)|)\chi(d)).$$

We obtain b) by simply multiplying by $|K|$.\\

Suppose now that b) holds. Again by problem 3.12 of \cite{Isaacs} we have that for every $h\in G$,
$$\frac{\chi(1)}{|G|}\sum_{g\in G}\chi(x(x^{h})^g)=\chi(x)\chi(x^h)=\chi(x)^2=\frac{\chi(1)}{|G|}(|{\rm \textbf{C}}_{G}(x)|\chi(1)+(|G|-|{\rm \textbf{C}}_{G}(x)|)\chi(d)).$$
Thus,
\begin{equation}\label{1}
\sum_{g\in G}\chi(x(x^{h})^g)=|{\rm \textbf{C}}_{G}(x)|\chi(1)+(|G|-|{\rm \textbf{C}}_{G}(x)|)\chi(d).
\end{equation}
Let $h\in G$ and suppose that $xx^h\not \in D$. We will prove that $xx^h=1$. By Eq.(\ref{1}) and taking into account the second ortogonality relation,

$$\sum_{g\in G}\sum_{\chi \in {\rm Irr}(G)}\chi(x(x^{h})^g)\overline{\chi(d)}=|{\rm \textbf{C}}_{G}(x)|\sum_{\chi \in {\rm Irr}(G)}\chi(1)\overline{\chi(d)}+(|G|-|{\rm \textbf{C}}_{G}(x)|)\sum_{\chi \in {\rm Irr}(G)}\chi(d)\overline{\chi(d)}=$$
$$(|G|-|{\rm \textbf{C}}_{G}(x)|)|{\rm \textbf{C}}_{G}(d)|.$$
On the other hand, since $xx^h\not \in D$, then
$$\sum_{g\in {\rm {\bf C}}_{G}(x^h)}\sum_{\chi \in {\rm Irr}(G)}\chi(x(x^{h})^g)\overline{\chi(d)}+\sum_{g\in G\setminus {\rm {\bf C}}_{G}(x^h)}\sum_{\chi \in {\rm Irr}(G)}\chi(x(x^{h})^g)\overline{\chi(d)}=$$
$$\sum_{g\in G\setminus {\rm {\bf C}}_{G}(x^h)}\sum_{\chi \in {\rm Irr}(G)}\chi(x(x^{h})^g)\overline{\chi(d)}.$$
Therefore,
$$\sum_{g\in G\setminus {\rm {\bf C}}_{G}(x^h)}\sum_{\chi \in {\rm Irr}(G)}\chi(x(x^{h})^g)\overline{\chi(d)}=(|G|-|{\rm \textbf{C}}_{G}(x)|)|{\rm \textbf{C}}_{G}(d)|.$$
Again by using the second ortogonality relation, we deduce that there are exactly $(|G|-|{\rm \textbf{C}}_{G}(x)|)$ elements $g\in G\setminus {\rm {\bf C}}_{G}(x^h)$ such that
$$\sum_{\chi \in {\rm Irr}(G)}\chi(x(x^{h})^g)\overline{\chi(d)}=|{\rm \textbf{C}}_{G}(d)|.$$
So, for every $g\in G\setminus {\rm {\bf C}}_{G}(x^h)$, we have $x(x^h)^g\in D$. Now, if we come back to Eq.(\ref{1}), we have
$$\sum_{g\in {\rm {\bf C}}_{G}(x^h)}\chi(x(x^{h})^g)+\sum_{g\in G\setminus {\rm {\bf C}}_{G}(x^h)}\chi(x(x^{h})^g)=|{\rm \textbf{C}}_{G}(x)|\chi(1)+(|G|-|{\rm \textbf{C}}_{G}(x)|)\chi(d)\, \, \, \, \, \, \forall \chi \in {\rm Irr}(G).$$
As a result,
$$|{\rm {\bf C}}_{G}(x^h)|\chi(xx^h)+(|G|-|{\rm \textbf{C}}_{G}(x)|)\chi(d)=|{\rm \textbf{C}}_{G}(x)|\chi(1)+(|G|-|{\rm \textbf{C}}_{G}(x)|)\chi(d)\, \, \, \, \, \, \forall \chi \in {\rm Irr}(G).$$
This implies that $\chi(xx^h)=\chi(1)$ for every $\chi \in {\rm Irr}(G)$, that is, $xx^h=1$. Therefore, we have proved that for every $h\in G$, either $xx^h=1$ or $xx^h\in D$, that is, $K^2\subseteq 1\cup D$. Since $K^2$ is a $G$-invariant set, the only possibilities are $K^2=D$ or $K^2=1\cup D$. However, if $K^2=D$, by Lemma 5, we have $\chi(x)^2=\chi(1)\chi(d)$ and by replacing in the equation of b), we get $|K|\chi(1)\chi(d)=\chi(1)^2+(|K|-1)\chi(1)\chi(d)$ for every $\chi \in$ Irr$(G)$. This forces that $\chi(d)=\chi(1)$ for every $\chi \in$ Irr$(G)$, so $d=1$, a contradiction. Then $K^2=1\cup D$, as wanted. \qed
\end{proof}

\begin{proof}[of Corollary C]
For every non central element $x\in G$, we know that $(x^G)^2=1\cup D$ for some conjugacy class $D$. Then $x^2\in \langle D\rangle$ and $\langle D\rangle$ is nilpotent by Theorem A. Thus, $x^2\in {\rm \textbf{F}}(G)$ for every $x\in G$. Consequently, $G/{\rm\textbf{F}}(G)$ is $2$-elementary abelian.\qed
\end{proof}

\begin{proof}[of Corollary D]
It may occur that $K^2$ is a conjugacy class and then, by applying Theorem 2, $\langle K\rangle$ is solvable. Otherwise, it happens  that either $K^2= A_1 \cup A_2 \cup \cdots \cup A_n$ or $K^2= A_1 \cup A_2 \cup \cdots \cup A_n \cup D$  with $A_i$ a central classs for every $i$ and $D$ a non-central class. If we consider $\overline{G}=G/{\rm \textbf{Z}}(G)$, it follows that either $\overline{K}^2=\overline{1}$ or $\overline{K}^2=\overline{1}\cup \overline{D}$. In the former case, $\langle \overline{K}\rangle$ is cyclic of order $2$ and as a consequence, $\langle K\rangle$ is solvable. In the second case, by applying Theorem A, $\langle \overline{K}\rangle=\langle K\rangle {\rm \textbf{Z}}(G)/{\rm \textbf{Z}}(G)$ is solvable, so $\langle K\rangle$ is solvable too.\qed
\end{proof}

\begin{proof}[of Corollary E]
We can easily prove that the hypotheses are inherited by factor groups and we work by induction on the order of $G$. If ${\rm\textbf{O}}_{\pi'}(G)\neq 1$ it easily follows by induction that $\overline{G}=G/{\rm \textbf{O}}_{\pi'}(G)$ is nilpotent. So we can assume that ${\rm\textbf{O}}_{\pi'}(G)=1$. For every $p\in \pi$, we choose $1\neq x_{p}\in {\rm \textbf{Z}}(P)$ for some $P\in$ Syl$_{p}(G)$. The hypotheses imply that $(x_{p}^{G})^2$ is a conjugacy class and, by applying Theorem 2, $|x_{p}^{G}|=|[x_{p}, G]|$ and hence, $[x_{p}, G]$ has $p'$-order. \\

Let $K_{p'}/[x_{p}, G]:={\rm \textbf{O}}_{\pi'}(G/[x_{p}, G]))$ which is a $p'$-group. Since $[x_{p}, G]$ is $p'$-group, then $K_{p'}\subseteq {\rm \textbf{O}}_{p'}(G)$. By induction, $$G/[x_{p}, G]/{\rm \textbf{O}}_{\pi'}(G/ [x_{p}, G])\cong G/K_{p'}$$ is nilpotent. Now, we consider the natural homomorphism $$\phi: G\longrightarrow G/K_{p_{1}'}\times \cdots \times G/K_{p_{s}'} $$ where $\pi=\lbrace p_{1}, \cdots, p_{s}\rbrace$. Since $\bigcap_{i=1}^sK_{p_{i}'}\subseteq \bigcap_{i=1}^{s}{\rm \textbf{O}}_{p_{i}'}(G)={\rm \textbf{O}}_{\pi'}(G)=1$, we conclude that $\phi$ is injective and thus, $G$ is nilpotent.\qed
\end{proof}

\begin{proof}[of Corollary F]
The hypotheses are inherited by quotients. Let us see that we can assume ${\rm \textbf{Z}}(G)=1$. Indeed, if we consider $\overline{G}=G/{\rm \textbf{Z}}(G)$ we have $${\rm \textbf{O}}_{2}(\overline{G}/ {\rm \textbf{F}}(\overline{G}))={\rm \textbf{O}}_{2}(\overline{G}/\overline{{\rm \textbf{F}}(G)})\cong{\rm \textbf{O}}_{2}(G/{\rm \textbf{F}}(G))=M/{\rm \textbf{F}}(G)\cong \overline{M}/\overline{{\rm \textbf{F}}(G)}.$$ If ${\rm \textbf{Z}}(G)>1$, arguing by induction on the order of $G$, we obtain that $\overline{G}/\overline{M}\cong G/M$ is nilpotent and then the theorem is proved. Thus, we can assume ${\rm \textbf{Z}}(G)=1$, as wanted.\\

We assume first that there exists a $2'$-element $x\in G\setminus {\rm \textbf{F}}(G)$ such that $(x^{G})^2$ is not a class. However, $(x^G)^2$ is union of conjugacy classes, all of them central except at most one. As ${\rm\textbf{Z}}(G)=1$, we have $(x^{G})^{2}=1\cup D$, where $D$ is a non-central class of $G$. By Theorem A, we conclude that $x^2\in \langle D\rangle \subseteq {\rm \textbf{F}}(G)$ and, since $\langle x\rangle=\langle x^2\rangle$, we get a contradiction.\\

Therefore, if we consider $\widehat{G}=G/{\rm \textbf{F}}(G)$, we can assume that every non-trivial $2'$-element  $\widehat{x}\in \widehat{G}$ satisfies that $(\widehat{x}^{\widehat{G}})^2$ is a conjugacy class of $\widehat{G}$. Observe that we can certainly assume that $x$ is a $2'$-element of $G$ such that $x\not\in {\rm \textbf{F}}(G)$ and that $(x^G)^2$ is a conjugacy class. We apply Corollary E with $\pi=\{2\}'$ in order to deduce that $\widehat{G}/{\rm \textbf{O}}_{2}(\widehat{G})$ is nilpotent, which implies that $G/M$ is nilpotent, and the proof is finished.\qed
\end{proof}

\end{document}